\documentclass[a4paper,leqno,12pt]{amsart}
\usepackage{amsmath,amstext,amssymb,amsopn,amsthm,enumitem,float,mathtools}
\usepackage{bbm}
\usepackage[utf8]{inputenc}
\usepackage{graphicx}
\usepackage{centernot}
\usepackage{tikz}
\usetikzlibrary{arrows}

\usepackage{tikz}
\usetikzlibrary{arrows}
\tikzset{dots/.append style={ultra thick, fill=none}}

\usepackage[hidelinks]{hyperref}
\usepackage{nameref}
\hypersetup{colorlinks = true, urlcolor = blue, linkcolor = blue, citecolor = red}

\newcommand{\NN}{\mathbb{N}}
\newcommand{\QQ}{\mathbb{Q}}
\newcommand{\ZZ}{\mathbb{Z}}
\newcommand{\RR}{\mathbb{R}}

\newcommand{\TT}{\mathbb{T}}

\newcommand{\tom}{\mathbb{T}^\omega}

\numberwithin{equation}{section}
\theoremstyle{plain}
\newtheorem{question}[equation]{Question}
\newtheorem{definition}[equation]{Definition}
\newtheorem{fact}[equation]{Fact}
\newtheorem{theorem}[equation]{Theorem}
\newtheorem{remark}[equation]{Remark}
\newtheorem{lemma}[equation]{Lemma}

\newtheorem{corollary}[equation]{Corollary}

\begin{document}
	
	\title[On the doubling condition in the infinite-dimensional setting]{On the doubling condition \\ in the infinite-dimensional setting}
	
	\author{Dariusz Kosz}
	\address{Dariusz Kosz (\textnormal{dkosz@bcamath.org}) \newline
	Basque Center for Applied Mathematics, 48009 Bilbao, Spain \& Wroc{\l}aw University of Science and Technology, 50-370 Wroc{\l}aw, Poland		
	}
	
	\begin{abstract} We present a systematic approach to the problem whether a topologically infinite-dimensional space can be made homogeneous in the Coifman--Weiss sense. The answer to the examined question is negative, as expected. Our leading representative of spaces with this property is $\tom = \TT \times \TT \times \cdots$ with the natural product topology. 
		
		\smallskip	
		\noindent \textbf{2020 Mathematics Subject Classification:} Primary 43A70, 54E99.
		
		\smallskip
		\noindent \textbf{Key words:} doubling condition, quasimetric space, infinite-dimensional torus.
	\end{abstract}
	
	\maketitle

\section{Introduction} \label{S1}

Given a nonempty topological space $(X,\mathcal T)$, its topological dimension ${\rm dim}(X)$ is the smallest number $n \in \NN \cup \{0\}$ with the property that each open cover $\mathcal B$ of $X$ has a~refinement\footnote{That is, a second open cover with all elements being subsets of elements of the first cover.} $\tilde{\mathcal B}$ such that each point $x \in X$ belongs to no more than $n+1$ elements of $\tilde{\mathcal B}$. If no such $n$ exists, then we put ${\rm dim}(X) = \infty$. 

The following note is devoted to explaining why a topologically infinite-dimensional space cannot be doubling. We shall refer to the doubling condition by using the notion of homogeneity in the Coifman--Weiss sense, see Definition~\ref{D2}.

\begin{theorem} \label{T1}
	Let $(X,\mathcal T)$ be a topological space. If ${\rm dim}(X) = \infty$, 
	then it is \textbf{not} possible to find a quasimetric $\rho$ and a Borel measure $\mu$ for which $\mathcal T_\rho = \mathcal T$ and $(X, \rho, \mu)$ is homogeneous in the Coifman--Weiss sense.
\end{theorem}

\noindent The same is true if the small ${\rm ind}(X)$ or large ${\rm Ind}(X)$ inductive dimension is used instead\footnote{For the definitions of the inductive topological dimensions, see for example \cite{En}.}. Indeed, homogeneous spaces are metrizable, see Facts~\ref{F1}~and~\ref{F2}, and separable, see \cite[Proposition~2.2]{St}, while all dimensions are topologically invariant and ${\rm ind}(X) = {\rm Ind}(X) = {\rm dim}(X)$ holds for separable metric spaces, see \cite[Preface]{En}.  

It should be emphasized that Theorem~\ref{T1} can be derived from general theory in just a~few lines, by using several results that are already known, as a black box, see the proof in Section~\ref{S2}.  
However, the problem lies at the intersection of different fields of research, and the solution relies on analytical, geometrical, and topological arguments that should be combined in the appropriate way. Therefore, we believe that it is worth making the topic more systematized by presenting a detailed approach which will be both elementary and instructive to the reader. We break down the original problem into several simpler subtasks, explain the reasons for making each reduction, and comment on possible obstacles or alternative paths along the way.   

Studying this kind of problem was originally motivated by a recent question by Roncal, related to the analysis on the infinite-dimensional torus. Since this space can be seen as a model example of $X$ from Theorem~\ref{T1}, we would like to look at the problem from the standpoint of this particular setting first, and only then pass to the general case.  

\subsection*{The infinite-dimensional torus $\tom$.} By $\tom$ we mean $\TT \times \TT \times \cdots$, that is, the product of countably many copies of the one-dimensional torus $\TT$. One can equip $\tom$ with the usual product topology $\mathcal T_{\tom}$ and the normalized Haar measure ${\rm d} x$\footnote{This is just the product of uniformly distributed probabilistic measures on $\TT$.} to make it a compact Hausdorff group and a metrizable probabilistic space. Then a lot of classical analysis can be developed in the context of $\tom$, including harmonic analysis which we focus on here.

Although the structure of $\tom$ seems nice at the first glance, careful looking at specific problems in this setting often leads to negative results or counterexamples to what we know from the Euclidean case $\RR^d$. To mention just a few such issues, one observes:
\begin{itemize}
	\item divergence of Fourier series of certain smooth functions \cite{FR19},
	\item no Lebesgue differentiation theorem for natural differentiation bases \cite{FR20, Ko},
	\item unboundedness of maximal operators \cite{Ko, KMPRR},
	\item problems with introducing a satisfactory theory of weights \cite{KMPRR}.
\end{itemize}

The instances we have chosen share one common feature. Precisely, they all originate in $\RR^d$-related questions to which answers are positive in the qualitative sense for each $d$ but also the worse in the quantitative sense the bigger $d$ is. In many cases the key reason for this phenomenon is the behavior of the so-called doubling condition. Indeed, although the estimate $|B(x,2r)| \leq C(d) |B(x,r)|$ is satisfied uniformly in $x \in \RR^d$ and $r \in (0,\infty)$ for $d$ fixed, the optimal constants $C(d) = 2^d$ grow exponentially with $d$. This fact usually becomes the main obstacle while trying to prove results with dimension-free bounds.

From this point of view one may expect that for $\tom$ the doubling condition is unlikely to hold as, loosely speaking, for each $d$ a piece of $\RR^d$ can be embedded in $\tom$. In this direction, the following question was asked by Roncal.
\begin{question} \label{Q1}
	Can one equip $\tom$ with a quasimetric $\rho$ and a measure $\mu$ so as
	to assure the doubling condition and, at the same time, keep the structure of $\tom$?
\end{question}

\noindent Several remarks regarding Question~\ref{Q1} are in order. 
\begin{enumerate}
	\item In the literature devoted to studying $\tom$, the most popular metric is given by  
	\[
	\rho_{\tom}(x,y) \coloneqq \sum_{n=1}^\infty \frac{\rho_\TT(x_n, y_n)}{2^n},\qquad x=(x_1, x_2, \dots),y=(y_1, y_2, \dots)\in\tom,
	\]
	where the following toric distance is used\footnote{In this formula we refer to the elements $x, y \in \TT$ as numbers belonging to $[0,1)$.}
	\[
	\rho_\TT(x,y) \coloneqq \min \{|x-y|, 1-|x-y|\}, \qquad x,y \in \TT. 
	\] 
	For $(\tom, \rho_{\tom}, {\rm d}x)$ the doubling condition fails to hold, see \cite[Chapter~2.3]{Fe}.  
	\item Bendikov in \cite[Remark~5.4.6]{Be} defines a family of metrics $\rho_A$ on $\tom$ by
	\[
	\rho_A(x,y) \coloneqq \Big( \sum_{n=1}^\infty a_n \rho^2_\TT(x_n,y_n) \Big)^{\frac{1}{2}}, \qquad A = (a_1, a_2, \dots) \in \mathcal A,
	\]
	where $\mathcal A$ is the space of all summable sequences with strictly positive entries. It was asked in \cite[Nota 2.34]{Fe} whether there exists an assumption on a sequence $A \in \mathcal A$ under which $(\tom, \rho_A, {\rm d}x)$ is a space of homogeneous type. This can be seen as a special case of Question~\ref{Q1}.
	
	\item The last part of Question~\ref{Q1} is essential, and omitting it would make the problem trivial. Indeed, in this case the following (not insightful) solution could be given:
	\[
	\textit{Yes, because there exist doubling spaces of the same cardinality as $\tom$.}
	\]
	For example, one could take $\RR$ with the standard distance and Lebesgue measure, and equip $\tom$ with $\rho$ and $\mu$ transferred from $\RR$ via a given bijection $\pi \colon \RR \to \tom$\footnote{In other words, one chooses $\rho$ and $\mu$ so that $\pi$ is a measure preserving isometry.}.
\end{enumerate}

We show that the answer to Question~\ref{Q1} is negative, as expected. This result has important consequences for the whole field of harmonic analysis on $\tom$, as it reveals that this subject goes beyond the theory of doubling spaces. In what follows, we present two theorems referring to either geometrical or topological structure of $\tom$.    
\begin{theorem} \label{T2}
	Suppose that $\rho$ is a bounded translation invariant quasimetric on $\tom$. Then it is \textbf{not} possible to find a measure $\mu$, defined on the $\sigma$-algebra generated by $\rho$, for which $(\tom, \rho, \mu)$ is homogeneous in the Coifman--Weiss sense. 
\end{theorem}

\begin{theorem} \label{T3}
	Suppose that $\rho$ is a quasimetric on $\tom$ such that $\mathcal T_\rho = \mathcal T_{\tom}$. Then it is \textbf{not} possible to find a Borel measure $\mu$ for which $(\tom, \rho, \mu)$ is homogeneous in the Coifman--Weiss sense.
\end{theorem}

\noindent Theorem~\ref{T2} is independent of Theorem~\ref{T1}, while Theorem~\ref{T3} is its special case with simpler proof. Both results answer the question in \cite[Nota~2.34]{Fe} in the negative. 

\subsection*{Homogeneous spaces.} Finally, we briefly recall the notion of homogeneity, see \cite{CW}. Alongside, we conduct a short discussion on quasimetrics, strongly inspired by \cite{St}.
\begin{definition} \label{D1}
	A \emph{quasimetric} on a nonempty set $X$ is a mapping $\rho \colon X \times X \to [0,\infty)$ satisfying the following conditions:
	\begin{itemize}
		\item $\rho(x,y) = 0$ if and only if $x = y$,
		\item $\rho(x,y) = \rho(y,x)$,
		\item $\rho(x,y) \leq K (\rho(x,z) + \rho(z,y))$ for some numerical constant $K \in [1,\infty)$.
	\end{itemize}
	If the last condition is satisfied with $K = 1$, then $\rho$ is called a \emph{metric}.
\end{definition}
\noindent There is a canonical way to introduce a topology on $X$ that corresponds to a given quasimetric $\rho$. Namely, for each $x \in X$ and $r \in (0,\infty)$ we denote 
\[
B_\rho(x,r) \coloneqq \{ y \in X : \rho(x,y) < r \},
\] 
the \emph{ball} centered at $x$ and of radius $r$. Then a set $G \subset X$ is said to be \emph{open} -- that is, $G \in \mathcal T_\rho$ -- if for each $x \in G$ there exists $r_x \in (0, \infty)$ such that $B_\rho(x, r_x) \subset G$. 
This definition turns out to be good for several reasons (commented in detail later on):
\begin{enumerate}[label=(\alph*)]
	\item \label{A} it extends the standard definition used in the metric case $K=1$,
	\item \label{B} it ensures that topology properties behave well under perturbations of $\rho$,
	\item \label{C} it always leads to a topology which is metrizable.   
\end{enumerate}
However, one needs to be careful because, quite surprisingly, in the case $K > 1$ it may happen that balls are not open or even Borel according to how $\mathcal T_\rho$ looks like.     

\begin{definition} \label{D2}
	Given a nonempty set $X$, a quasimetric $\rho$, and a Borel measure $\mu$, we call the space $(X,\rho,\mu)$ \emph{homogeneous in the Coifman--Weiss sense} if $\mu(B_\rho) \in (0, \infty)$ holds for all balls $B_\rho \subset X$\footnote{In particular, it is assumed that balls are $\mu$-measurable sets.}, and there exists a numerical constant $C \in [1, \infty)$ such that
	\[
	\mu( B_\rho(x,2r)) \leq C \mu(B_\rho(x,r))
	\quad \text{for all} \quad 
	x \in X, \, r \in (0,\infty).
	\] 
	If this last condition holds, then we say that $\mu$ is \emph{doubling} with respect to $\rho$. 
\end{definition}
\noindent Notice that in general, under the assumption that all balls are measurable, the doubling condition leads to the following trichotomy:
\begin{itemize}
	\item $\mu(B_\rho) = 0$ for all balls $B_\rho \subset X$ and, consequently, $\mu(X) = 0$,
	\item $\mu(B_\rho) \in (0,\infty)$ for all balls $B_\rho \subset X$,
	\item $\mu(B_\rho) = \infty$ for all balls $B_\rho \subset X$.
\end{itemize}
Thus, the condition $\mu(B_\rho) \in (0, \infty)$ in Definition~\ref{D2} excludes only trivial examples.

\subsection*{Acknowledgments} 
The author is indebted to Luz Roncal for drawing his attention to a very interesting problem discussed in this article, as well as for many long discussions which led to a deeper understanding of the subject and significant improvements in the presentation of the results. 

The author was supported by the Basque Government through the BERC 2022-2025 program, by the Spanish State Research Agency through BCAM Severo Ochoa excellence accreditation SEV-2017-2018, and by the Foundation for Polish Science through the START Scholarship.

\section{Proofs of Theorems~\ref{T1},~\ref{T2},~and~\ref{T3}} \label{S2}

\subsection{Analysis: from quasimetric to metric spaces.}

Instead of dealing directly with the problems stated in Section~\ref{S1}, we opt to make some reductions in advance, in order to get rid of several technicalities such as measurability of balls. The first reduction refers to the `metamathematical principle' \cite[Section~3]{St} which says that many quasimetric-related questions can be boiled down to the metric case.

One reason the definition of quasimetric is convenient to use is that if $\rho$ is a quasimetric and $\tilde \rho$ is symmetric and comparable to $\rho$, then $\tilde \rho$ is a quasimetric as well. For metrics the corresponding statement is not true. However, the strength of this flexibility sometimes turns into weakness. Indeed, the definition of topology using balls is perfectly suited to the metric case and we pay a certain cost to ensure \ref{A}. If $K=1$, then the triangle inequality provides that, for an arbitrary reference point $x$ and two points $y,z$ lying close to each other, the distances $\rho(x,y), \rho(x,z)$ are similar. Precisely, we have $| \rho(x,y) - \rho(x,z)| \leq \rho(y,z)$. Thus, if $y \in B_\rho(x,r)$, then also $B_\rho(y,\tilde r) \subset B_\rho(x,r)$ for some appropriately chosen $\tilde r$, so that the ball $B_\rho(x,r)$ is open. This is not true in general if $K > 1$. To see this, take $\RR$ with the standard metric $\rho_{\RR}(x,y) = |x-y|$ and modify it putting
\begin{itemize}
	\item $\tilde \rho_{\RR}(x,y) = 2 \rho_{\RR}(x,y)$ if one of the points is $0$ while the other one belongs to a given set $E \subset (1,2)$,
	\item $\tilde \rho_{\RR}(x,y) = \rho_{\RR}(x,y)$ otherwise.
\end{itemize}
In view of the former discussion $\tilde \rho_{\RR}$ is a quasimetric. Moreover, looking at the notion of convergence, we would expect it to generate the same topology on $\RR$ as the standard one. However, the exact forms of the balls $B_{\tilde \rho_{\RR}}(0,r)$ with $r \in (1,4)$ are strongly dependent on the form of $E$ itself, while  this set is arbitrary. In particular, one can choose $E$ for which none of these balls is Borel\footnote{See \cite[Example~1.1]{St} for a simple example of quasimetric space such that all balls fail to be Borel.}.

Nonetheless, once we realize that, instead of balls, the topology $\mathcal T_\rho$ is what we should look at, things start to get more optimistic. To see this, we need the following definition.            

\begin{definition} \label{D3}
	Two quasimetrics on $X$, $\rho_1$ and $\rho_2$, are called \emph{equivalent} if there exists a numerical constant $M \in [1,\infty)$ such that $M^{-1} \rho_1 \leq \rho_2 \leq M \rho_1$. 
\end{definition}
\noindent It is easy to verify the result below, which one can relate to \ref{B}. 

\begin{fact} \label{F1}
	If $\rho_1$ and $\rho_2$ are equivalent, then $\mathcal T_{\rho_1} = \mathcal T_{\rho_2}$. Also, for a quasimetric $\rho$ and $\alpha \in (0,\infty)$, the mapping $\rho^\alpha$ defines a quasimetric such that $\mathcal T_{\rho^\alpha} = \mathcal T_{\rho}$.  
\end{fact}
\noindent The next fact, which justifies \ref{C}, can be used to reduce our problems to the metric case.

\begin{fact} \label{F2}
	Consider a quasimetric $\rho$ on $X$ and take $q \in (0,1]$ satisfying $(2K)^q = 2$. Then
	\[
	\rho_q(x,y) \coloneqq \inf \Big\{ \sum_{j=1}^{n} \rho(x_j, x_{j-1})^q : x=x_0, x_1, \dots, x_n=y, \, n \in \NN \Big\}
	\]
	determines a metric on $X$ which is equivalent to $\rho^q$. Precisely, one has $\rho_q \leq \rho^q \leq 4 \rho_q$.  
\end{fact}

\noindent The proof of Fact~\ref{F2} can be found in \cite[Proposition]{PS}, see also \cite{AIN}. We now explain briefly what is the motivation behind such a definition of $\rho_q$. If $K > 1$, then $\rho(x,y) > \rho(x,z) + \rho(z,y)$ can happen. Thus, to assure the triangle inequality we would like to make the distance between $x$ and $y$ not larger than the right hand side. The same applies to $\rho(x,z), \rho(z,y)$ so we eventually take into account all finite chains going from $x$ to $y$. But then, as the number of intermediate points goes to infinity, the corresponding expressions may go to zero\footnote{For example, if $\tilde \rho(x,y) \coloneqq (x-y)^2$, $x,y \in \RR$, then $\lim_{n \to \infty} \tilde \rho(0, \frac{1}{n}) + \tilde \rho(\frac{1}{n},\frac{2}{n}) + \cdots + \tilde \rho(\frac{n-1}{n}, 1) = 0$.}. Hence, we need to adjust our original idea and, as it turns out, penalizing long chains by using $q$ close to zero does the job perfectly.

\begin{corollary} \label{C1}
	Regarding Theorems~\ref{T1},~\ref{T2},~and~\ref{T3} it is enough to consider metrics. 
\end{corollary}

\noindent Indeed, by using Facts~\ref{F1}~and~\ref{F2}, one can verify that if $(\tom,\rho,\mu)$ is a quasimetric space which is homogeneous in the Coifman--Weiss sense, then $(\tom,\rho_q, \mu)$ is a homogeneous metric space that enjoys the same topology. Also, if $\rho$ is bounded and translation invariant, then so is $\rho_q$.

From now on, we can concentrate solely on metrics. However, to satisfy the reader's curiosity, we shall comment on which results have their quasimetric analogues. 

\subsection{Geometry: from doubling to geometrically doubling spaces} 

Our next goal is to show that yet another important reduction can be made. Namely, although both $\rho$ and $\mu$ are involved in verifying whether $(X, \rho, \mu)$ is homogeneous or not, it is actually the metric that plays the more important role here. 

It is clear that if $\mu$ is doubling with respect to $\rho$ and the second option in the said trichotomy occurs, then one should not be able to find arbitrarily many disjoint balls of radius $\frac{r}{2}$ centered at points $y \in B_\rho(x,2r)$. Indeed, if that would be the case, then at least one of these balls, say $B_\rho(y_0,\frac{r}{2})$, should have very small measure compared to $\mu(B_\rho(x,4r))$\footnote{This follows because there are many disjoint balls, each of them satisfying $B_\rho(y,\frac{r}{2}) \subset B_\rho(x,4r)$.}, and the doubling condition would fail for one of the balls $B_\rho(y_0,\frac{r}{2}), B_\rho(y_0,r), B_\rho(y_0,2r)$.

The discussion above motivates the following definition.         
	
\begin{definition} \label{D4}
	A quasimetric space $(X,\rho)$ is called \emph{geometrically doubling} if there exists a number $N \in \NN$ such that every ball $B_\rho(x,2r)$ can be covered by no more than $2^N$ balls of radius $r$. In this case, we also say that $\rho$ is \emph{geometrically doubling}. 
\end{definition} 

\noindent It turns out that, in some sense, failing to be geometrically doubling is the only obstacle that prevents a given space from becoming homogeneous after a suitable choice of $\mu$. 

\begin{fact} \label{F3}
	If a metric space $(X,\rho,\mu)$ is homogeneous in the Coifman--Weiss sense, then $\rho$ is geometrically doubling. Conversely, if $\rho$ is a geometrically doubling metric on $X$, then there exists a Borel measure $\mu$ such that $(X,\rho,\mu)$ is homogeneous in the Coifman--Weiss sense, provided that $(X, \rho)$ is complete.  
\end{fact}

\noindent Indeed, the first part of Fact~\ref{F3} is a known fact mentioned by the authors in \cite{CW}, see also \cite{Hy}. Precisely, if $\rho$ is not geometrically doubling, then for each $M \in \NN$ there exist a ball $B_\rho(x,2r)$ and points $y_1, \dots, y_{M} \in B_\rho(x,2r)$ such that $\rho(y_i, y_j) \geq r$ if $i \neq j$, so that the balls $B_\rho(y_1,\frac{r}{2}), \dots, B_\rho(y_{M},\frac{r}{2})$ are disjoint. Then the doubling condition cannot hold in view of the previous discussion. The reverse part is harder and its proof can also be found in \cite{LS}, see also . The quasimetric analogue of Fact~\ref{F3} is also true\footnote{In the reverse part, we additionally assume that $\rho$ is such that all balls are Borel.}. Finally, in general the completeness assumption cannot be ignored\footnote{To see this, consider $\QQ$ and $\rho_\RR$ restricted to $\QQ \times \QQ$, as mentioned in \cite{St}.}.

\begin{corollary} \label{C2}
	Regarding Theorems~\ref{T1},~\ref{T2},~and~\ref{T3} one only needs to look for geometrically doubling metrics satisfying the desired properties. 
\end{corollary}

\noindent Indeed, this follows clearly by combining Corollary~\ref{C1} and Fact~\ref{F3}. Precisely, we expect negative answers so it suffices to show that each metric $\rho$ which is either bounded and translation invariant (Theorem~\ref{T2}) or such that $\mathcal T_\rho$ coincides with the given topology (Theorems~\ref{T1}~and~\ref{T3}) cannot be geometrically doubling.     

To use the geometrical doubling property, we introduce the concept of $r$-separated sets.

\begin{definition} \label{D5}
	For a nonempty quasimetric space $(X,\rho)$ we say that a given subset $E \subset X$ is \emph{$r$-separated}, $r \in (0,\infty)$, if $\rho(x,y) \geq r$ for all distinct $x,y \in E$. We denote by $\aleph(X,\rho,r)$ the biggest number $n \in \NN$ such that there exists at least one $r$-separated set with $n$ elements. If arbitrarily large $r$-separated sets can be found, then we put $\aleph(X,\rho,r) = \infty$.
\end{definition}

\noindent The following lemma will be very helpful later on.

\begin{lemma} \label{L1}
	Let $(X,\rho)$ be a bounded metric space. If $\rho$ is geometrically doubling with some $N \in \NN$, then there exists $C \in (0,\infty)$ such that $\aleph(X,\rho,2^{-l}) \leq C 2^{Nl}$ for all $l \in \NN$.  
\end{lemma}

\begin{proof}
	Take $L \in \ZZ$ such that $\sup_{x,y \in X} \rho(x,y) < 2^L$.  
	Then for an arbitrary reference point $x \in X$ we have $B_\rho(x,2^L) = X$ and iterating the covering procedure we conclude that for each $l \in \NN$ the space $X$ can be covered by $2^{Nl}$ balls of radius $2^{L-l}$, so that $\aleph(X,\rho,2^{L-l+1}) \leq 2^{Nl}$ holds\footnote{To see this, notice that if $\rho(x,y) \geq 2r$, then there is no ball of radius $r$ containing both $x$ and $y$.}. A suitable reparametrization gives the thesis with some $C$ depending on $L,N$.   
\end{proof}
\noindent A quasimetric version of Lemma~\ref{L1} is also true, but with $C2^{Ml}$ instead of $C2^{Nl}$, where $C$ depends on $K,L,N$, while $M$ depends only on $K,N$.  

We are ready to prove the first of the two $\tom$-related theorems.

\begin{proof}[Proof of Theorem~\ref{T2}.]
	Suppose that $\rho$ is a bounded translation invariant metric on $\tom$. For each $n,j \in \NN$ consider the set
	\[
	E_{n,j} = \Big\{ (x_1, \dots, x_n, 0, 0, \dots) \in \tom : x_1, \dots, x_n \in \big \{ 0 \cdot 2^{-j}, 1 \cdot 2^{-j}, \dots , (2^{j}-1) \cdot 2^{-j} \big \} \Big \}.
	\]
	Then $E_{n,j}$ has precisely $2^{nj}$ elements, and it is $r_{n,j}$-seperated with $r_{n,j}$ satisfying 
	\[
	r_{n,j} = \min_{x,y \in E_{n,j} : x \neq y } \rho(x,y)
	=
	\min_{z \in E_{n,j} \setminus \{ {\bf 0} \} } \rho( {\bf 0},z),
	\]
	where ${\bf 0} = (0,0, \dots) \in \tom$ is the neutral element of the group. Indeed, the last equality follows, since $\rho$ is translation invariant and $E_{n,j}$ is a subgroup of $\tom$. 
	
	Let us now observe that if $z \in E_{n,j+1} \setminus \{ {\bf 0} \}$ for some $j \in \NN$, then either $z \in E_{n,1} \setminus \{ {\bf 0} \}$ or $2z \in E_{n,j} \setminus \{ {\bf 0} \}$, see Figure~\ref{Fig1}. In the first case $\rho( {\bf 0},z) \geq r_{n,1}$, while in the second one, by translation invariance and the triangle inequality, one has $\rho({\bf 0}, z) = \frac{1}{2} (\rho( {\bf 0}, z) + \rho(z, 2z)) \geq \frac{1}{2} \rho( {\bf 0}, 2z) \geq \frac{r_{n,j}}{2}$. 
	Thus, $r_{n,j+1} \geq \min \{ r_{n,1}, \frac{r_{n,j}}{2} \}$ and denoting $C_n = r_{n,1}$ we conclude that $r_{n,j} \geq C_n 2^{-j+1}$ for each $j \in \NN$, so that $\aleph(\tom, \rho, C_n 2^{-j+1}) \geq 2^{nj}$. 
	
	Since both $n,j$ may be arbitrarily large, one can use Lemma~\ref{L1} to deduce that $\rho$ cannot be geometrically doubling. Indeed, there is no $N \in \NN$ such that $\aleph(\tom,\rho,2^{-l}) \leq C2^{Nl}$ holds for all $l \in \NN$ with some $C \in (0,\infty)$, as otherwise one gets a contradiction by taking any $n$ greater than $N$ and sufficiently large $j$ depending on $N,C,C_n$.       
\end{proof}

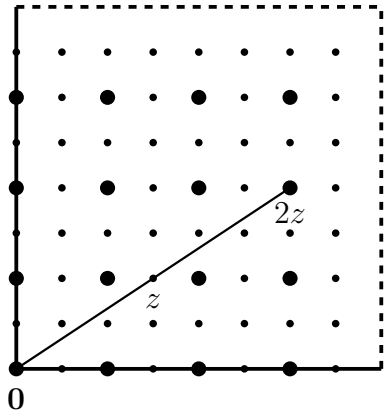
\begin{figure}[H]
	\begin{tikzpicture}
	[
	scale=0.6,
	important line/.style={thick},
	dashed line/.style={dashed, thin},
	every node/.style={color=black,circle,fill}
	]
	
	\node[label={[yshift=-25pt]$\bf 0$},inner sep=2pt] (im) at (0,0)  {};
	\node[label=$ $,inner sep=2pt] (im) at (0,2)  {};
	\node[label=$ $,inner sep=2pt] (im) at (0,4)  {};
	\node[label=$ $,inner sep=2pt] (im) at (0,6)  {};
	
	\node[label=$ $,inner sep=2pt] (im) at (2,0)  {};
	\node[label=$ $,inner sep=2pt] (im) at (2,2)  {};
	\node[label=$ $,inner sep=2pt] (im) at (2,4)  {};
	\node[label=$ $,inner sep=2pt] (im) at (2,6)  {};
	
	\node[label=$ $,inner sep=2pt] (im) at (4,0)  {};
	\node[label=$ $,inner sep=2pt] (im) at (4,2)  {};
	\node[label=$ $,inner sep=2pt] (im) at (4,4)  {};
	\node[label=$ $,inner sep=2pt] (im) at (4,6)  {};
	
	\node[label=$ $,inner sep=2pt] (im) at (6,0)  {};
	\node[label=$ $,inner sep=2pt] (im) at (6,2)  {};
	\node[label={[yshift=-25pt]$2z$},inner sep=2pt] (im) at (6,4)  {};
	\node[label=$ $,inner sep=2pt] (im) at (6,6)  {};
	
	\node[label=$ $,inner sep=1pt] (im) at (0,1)  {};
	\node[label=$ $,inner sep=1pt] (im) at (0,3)  {};
	\node[label=$ $,inner sep=1pt] (im) at (0,5)  {};
	\node[label=$ $,inner sep=1pt] (im) at (0,7)  {};
	
	\node[label=$ $,inner sep=1pt] (im) at (2,1)  {};
	\node[label=$ $,inner sep=1pt] (im) at (2,3)  {};
	\node[label=$ $,inner sep=1pt] (im) at (2,5)  {};
	\node[label=$ $,inner sep=1pt] (im) at (2,7)  {};
	
	\node[label=$ $,inner sep=1pt] (im) at (4,1)  {};
	\node[label=$ $,inner sep=1pt] (im) at (4,3)  {};
	\node[label=$ $,inner sep=1pt] (im) at (4,5)  {};
	\node[label=$ $,inner sep=1pt] (im) at (4,7)  {};
	
	\node[label=$ $,inner sep=1pt] (im) at (6,1)  {};
	\node[label=$ $,inner sep=1pt] (im) at (6,3)  {};
	\node[label=$ $,inner sep=1pt] (im) at (6,5)  {};
	\node[label=$ $,inner sep=1pt] (im) at (6,7)  {};
	
	\node[label=$ $,inner sep=1pt] (im) at (1,0)  {};
	\node[label=$ $,inner sep=1pt] (im) at (1,1)  {};
	\node[label=$ $,inner sep=1pt] (im) at (1,2)  {};
	\node[label=$ $,inner sep=1pt] (im) at (1,3)  {};
	\node[label=$ $,inner sep=1pt] (im) at (1,4)  {};
	\node[label=$ $,inner sep=1pt] (im) at (1,5)  {};
	\node[label=$ $,inner sep=1pt] (im) at (1,6)  {};
	\node[label=$ $,inner sep=1pt] (im) at (1,7)  {};
	
	\node[label=$ $,inner sep=1pt] (im) at (3,0)  {};
	\node[label=$ $,inner sep=1pt] (im) at (3,1)  {};
	\node[label={[yshift=-20pt]$z$},inner sep=1pt] (im) at (3,2)  {};
	\node[label=$ $,inner sep=1pt] (im) at (3,3)  {};
	\node[label=$ $,inner sep=1pt] (im) at (3,4)  {};
	\node[label=$ $,inner sep=1pt] (im) at (3,5)  {};
	\node[label=$ $,inner sep=1pt] (im) at (3,6)  {};
	\node[label=$ $,inner sep=1pt] (im) at (3,7)  {};
	
	\node[label=$ $,inner sep=1pt] (im) at (5,0)  {};
	\node[label=$ $,inner sep=1pt] (im) at (5,1)  {};
	\node[label=$ $,inner sep=1pt] (im) at (5,2)  {};
	\node[label=$ $,inner sep=1pt] (im) at (5,3)  {};
	\node[label=$ $,inner sep=1pt] (im) at (5,4)  {};
	\node[label=$ $,inner sep=1pt] (im) at (5,5)  {};
	\node[label=$ $,inner sep=1pt] (im) at (5,6)  {};
	\node[label=$ $,inner sep=1pt] (im) at (5,7)  {};
	
	\node[label=$ $,inner sep=1pt] (im) at (7,0)  {};
	\node[label=$ $,inner sep=1pt] (im) at (7,1)  {};
	\node[label=$ $,inner sep=1pt] (im) at (7,2)  {};
	\node[label=$ $,inner sep=1pt] (im) at (7,3)  {};
	\node[label=$ $,inner sep=1pt] (im) at (7,4)  {};
	\node[label=$ $,inner sep=1pt] (im) at (7,5)  {};
	\node[label=$ $,inner sep=1pt] (im) at (7,6)  {};
	\node[label=$ $,inner sep=1pt] (im) at (7,7)  {};
	
	\draw[important line, line width=0.5mm] (0,0) -- (0,8);
	\draw[important line, line width=0.5mm] (0,0) -- (8,0);
	\draw[dashed line, line width=0.5mm] (8,0) -- (8,8);
	\draw[dashed line, line width=0.5mm] (0,8) -- (8,8);
	
	\draw[important line, line width=0.3mm] (0,0) -- (6,4);
	
	\end{tikzpicture}
	\caption{Visualization of the sets $E_{n,j}$, $j \in \NN$, for $n=2$. Thick dots and all dots correspond to the sets $E_{2,2}$ and $E_{2,3}$, respectively. Given $z \in E_{2,3}$, we have $2z \in E_{2,2}$ and $2z={\bf 0} \iff z \in E_{2,1}$.}
	\label{Fig1}
\end{figure}

\noindent At the expense of additional technical difficulties, one can show Theorem~\ref{T2} directly for all bounded and translation invariant quasimetrics, modifying the proof presented above.

\subsection{Topology: from Hausdorff to topological dimension}

Next we prove Theorem~\ref{T3}. Here we use the following classical result that can be seen as a special case of the Brouwer fixed-point theorem or a~multidimensional variant of the Darboux theorem. 

\begin{fact}[Poincar\'e--Miranda theorem]\label{F4}
	  For $n \in \NN$ let $f_1, \dots, f_n$ be continuous functions defined on $[0,1]^n$. Assume that for each $i \in \{1, \dots, n\}$ and $(x_1, \dots, x_n) \in [0,1]^n$ there exists $a_i \in \RR$ such that $f_i(x) \leq a_i$ if $x_i = 0$ and $f_i(x) \geq a_i$ if $x_i = 1$. Then there exists $x^* \in [0,1]^n$ such that $(f_1(x^*), \dots, f_n(x^*)) = (a_1, \dots, a_n)$.
\end{fact}
\noindent Thanks to Fact~\ref{F4} we can adapt the idea behind the previous proof to the case of metrics which are not necessarily translation invariant.  

\begin{proof}[Proof of Theorem~\ref{T3}] 
	Suppose that $\rho$ is such that $\mathcal T_\rho = \mathcal T_{\tom}$. Then $\rho$ is bounded because $(\tom, \rho)$ is compact. Moreover, $E \subset \tom$ is $\mathcal T_\rho$-compact if and only if it is $\mathcal T_{\tom}$-closed.
	
	For each $n \in \NN$ consider the set
	\[
	E_n \coloneqq \big \{ (x_1, \dots, x_n, 0,0, \dots) \in \tom : (x_1, \dots, x_n) \in [0,\tfrac{1}{2}]^n \big \}
	\]
	which will play the role of the cube $[0,1]^n$ from Fact~\ref{F4}.
	Denote also, for $i \in \{1, \dots, n\}$,
	\[
	E_{n,i}^- \coloneqq \{ x \in E_n : x_i = 0 \}
	\quad \text{and} \quad
	E_{n,i}^+ \coloneqq \{ x \in E_n : x_i = \tfrac{1}{2} \},
	\]
	and set
	\[
	C_n \coloneqq \inf \big\{ \rho(x,y) : x \in E_{n,i}^-, \,  y \in E_{n,i}^+ \text{ for some }  i \in \{1, \dots, n\} \big\}.
	\]
	Since $E_{n,i}^-, E_{n,i}^+$ are compact and $(x,y) \mapsto \rho(x,y)$ is continuous\footnote{Here it is important that $\rho$ is a metric.} on $\tom \times \tom$, we have $C_n \in (0, \infty)$. Define auxiliary functions
	\[
	f_{n,i}(x) \coloneqq \inf_{ y \in E_{n,i}^-} \rho(x,y), \qquad x \in E_n.  
	\]
	Using compactness again we deduce that each $f_{n,i}$ is continuous\footnote{Indeed, assuming $f_{n,i}(x) \geq f_{n,i}(x')$, we get
	$0 \leq f_{n,i}(x) - f_{n,i}(x') \leq \rho(x,y^*) - \rho(x',y^*) \leq \rho(x,x')$, by taking $y^* \in E_{n,i}^-$ for which the value $f_{n,i}(x)$ is attained. Again, it is important here that $\rho$ is a metric.}. Moreover, $f_{n,i}(x) = 0$ for $x \in E_{n,i}^-$ and $f_{n,i}(x) \geq C_n$ for $x \in E_{n,i}^+$. 
	
	Next, choose $j \in \NN$ and take $v = (v_1, \dots, v_n) \in \big\{ \frac{C_n}{2^j}, \frac{2C_n}{2^j}, \dots, \frac{2^jC_n}{2^j} \big\}^n$. By Fact~\ref{F4} there exists $x_v \in E_n$ such that $(f_{n,1}(x_v), \dots, f_{n,n}(x_v)) = v$. We shall show that the set
	\[
	E_{n,j} \coloneqq \Big\{ x_v : v \in \big\{ \tfrac{C_n}{2^j}, \tfrac{2C_n}{2^j}, \dots, \tfrac{2^{j}C_n}{2^j} \big\}^n \Big\}
	\] 
	of cardinality $2^{nj}$ is $\frac{C_n}{2^j}$-separated so that $\aleph(\tom, \rho, \frac{C_n}{2^j}) \geq 2^{nj}$ holds. To this end, let $x_v, x_{v'} \in E_{n,j}$ correspond to distinct vectors $v, v'$ and assume that $v_{i_0}' > v_{i_0}$ for some ${i_0} \in \{1,\dots, n\}$. Then $f_{n,i_0}(x_{v'}) \geq f_{n,i_0}(x_v) + \frac{C_n}{2^j}$ by the definition of $f_{n,i_0}$, while the triangle inequality gives $f_{n,i_0}(x_{v'}) \leq f_{n,i_0}(x_v) + \rho(x_v, x_{v'})$, see Figure~\ref{Fig2}. Thus, $\rho(x_v, x_{v'}) \geq \frac{C_n}{2^j}$. 
	
	Both $n,j$ may be arbitrarily large so one can use Lemma~\ref{L1} to deduce that $\rho$ cannot be geometrically doubling. Indeed, there is no $N \in \NN$ such that $\aleph(\tom,\rho,2^{-l}) \leq C 2^{Nl}$ holds for all $l \in \NN$ with some $C \in (0,\infty)$, as otherwise one gets a contradiction by taking any $n$ greater than $N$, and sufficiently large $j$ depending on $N,C,C_n$. 
\end{proof}

\begin{figure}[H]
	\begin{tikzpicture}
	[
	scale=0.6,
	important line/.style={thick},
	dashed line/.style={dashed, thin},
	every node/.style={color=black,circle,fill}
	]
	
	\node[label=$ $,inner sep=2pt] (im) at (1.125,1)  {};
	\node[label={[xshift=-8pt, yshift=-8pt]$x_v$},inner sep=2pt] (im) at (1.375,3)  {};
	\node[label=$ $,inner sep=2pt] (im) at (1.625,5)  {};
	\node[label=$ $,inner sep=2pt] (im) at (1.875,7)  {};
	
	\node[label={[xshift=10pt, yshift=-10pt]$x_{v'}$},inner sep=2pt] (im) at (3.875,1)  {};
	\node[label=$ $,inner sep=2pt] (im) at (3.625,3)  {};
	\node[label=$ $,inner sep=2pt] (im) at (3.375,5)  {};
	\node[label=$ $,inner sep=2pt] (im) at (3.125,7)  {};
	
	\node[label=$ $,inner sep=2pt] (im) at (5.125,1)  {};
	\node[label=$ $,inner sep=2pt] (im) at (5.375,3)  {};
	\node[label=$ $,inner sep=2pt] (im) at (5.625,5)  {};
	\node[label=$ $,inner sep=2pt] (im) at (5.875,7)  {};
	
	\node[label=$ $,inner sep=2pt] (im) at (7.25,1)  {};
	\node[label=$ $,inner sep=2pt] (im) at (7.75,3)  {};
	\node[label=$ $,inner sep=2pt] (im) at (7.75,5)  {};
	\node[label=$ $,inner sep=2pt] (im) at (7.25,7)  {};
	
	\draw[important line, line width=0.5mm] (0,0) -- (0,8);
	\draw[important line, line width=0.5mm] (0,0) -- (8,0);
	\draw[important line, line width=0.5mm] (8,0) -- (8,8);
	\draw[important line, line width=0.5mm] (0,8) -- (8,8);
	
	\node[label=$E_{2,1}^-$,color=white] (im) at (-1,2.5)  {};
	\node[label=$E_{2,1}^+$,color=white] (im) at (9,2.5)  {};
	
	\node[label=$E_{2,2}^-$,color=white] (im) at (4,-2.5)  {};
	\node[label=$E_{2,2}^+$,color=white] (im) at (4,7.5)  {};
	
	\draw[dashed line, line width=0.5mm] (0,1) -- (8,1);
	\draw[dashed line, line width=0.5mm] (0,3) -- (8,3);
	\draw[dashed line, line width=0.5mm] (0,5) -- (8,5);
	\draw[dashed line, line width=0.5mm] (0,7) -- (8,7);
	
	\draw[dashed line, line width=0.5mm] (1,0) -- (2,8);
	\draw[dashed line, line width=0.5mm] (4,0) -- (3,8);
	\draw[dashed line, line width=0.5mm] (5,0) -- (6,8);
	\draw[dashed line, line width=0.5mm] (7,0) -- (8,4);
	\draw[dashed line, line width=0.5mm] (7,8) -- (8,4);
	
	\node[label={[xshift=-7pt, yshift=-13pt]$y^*$},inner sep=1pt] (im) at (0,2)  {};
	
	\draw[important line, line width=0.3mm] (0,2) -- (3.875,1);
	\draw[important line, line width=0.3mm] (1.375,3) -- (3.875,1);
	\draw[important line, line width=0.3mm] (1.375,3) -- (0,2);
	
	\end{tikzpicture}
	\caption{Visualization of the `cube' $E_n$ for $n=2$. Thick dots are the elements of the set $E_{2,j}$ with $j=2$, while dashed lines represent the corresponding level sets of the functions $f_{2,1}, f_{2,2}$ (this is an oversimplified scheme, as in general the structure of the level sets may be much more complicated). We pick two points $x_v, x_{v'} \in E_{2,2}$ corresponding to vectors $v,v'$ such that $v'_{i_0} > v_{i_0}$ for $i_0 = 1$. By $y^*$ we denote the point from $E_{2,1}^-$ for which $f_{2,1}(x_v)$ is attained. Then $f_{2,1}(x_{v'}) \leq \rho(x_{v'},y^*) \leq f_{2,1}(x_v) + \rho(x_v, x_{v'})$.}
	\label{Fig2}
\end{figure}
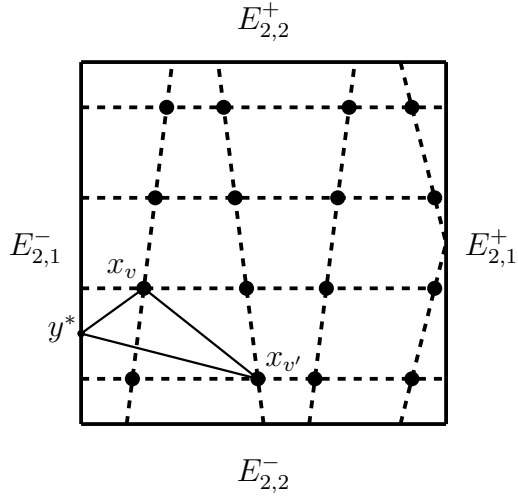
\noindent This time it was crucial that only metrics, not quasimetrics, were considered in the proof.
  
It remains to show Theorem~\ref{T1}. To this end, let us recall the concept of the Hausdorff dimension. Given a metric space $(X,\rho)$, for each $E \subset X$ we define
\[
\mathcal H^d(E) \coloneqq \lim_{\delta \downarrow 0} \Big( \inf
\Big\{  \sum_{i=1}^\infty \big({\rm diam} (U_i) \big)^d : E \subset \bigcup_{i=1}^\infty U_i, \, {\rm diam} (U_i) < \delta  \Big \} \Big),
\qquad d \in [0,\infty), 
\]
and put ${\rm dim}_{\mathcal H}(E) \coloneqq \inf \{ d \in [0,\infty) : \mathcal H^d(E) = 0\}$\footnote{We use the convention ${\rm dim}_{\mathcal H}(E) = \infty$ if the infimum is taken over the empty set.}. The proof of Lemma~\ref{L1} reveals that if $(X,\rho)$ is geometrically doubling with some $N \in \NN$, and $x \in X$ is any reference point, then ${\rm dim}_{\mathcal H}(X) = \lim_{r \to \infty} {\rm dim}_{\mathcal H}(B_\rho(x,r)) \leq N$. Similarly, the proof of Theorem~\ref{T3} hints that $[0,1]^n$ equipped with any metric generating the standard topology should have Hausdorff dimension at least $n$. The latter is a special case of the following general result.

\begin{fact}\label{F5}
	Let $(X, \rho)$ be a separable metric space. Then ${\rm dim}(X) \leq {\rm dim}_{\mathcal H}(X)$. 
\end{fact}

\noindent Indeed, ${\rm dim}(X) = {\rm ind}(X)$ follows for separable metric spaces, see \cite[Preface]{En}, while ${\rm ind}(X) \leq {\rm dim}_{\mathcal H}(X)$ follows for metric spaces, see \cite[Section~3.1]{Ed}. For separable quasimetric spaces ${\rm dim}(X) \leq \frac{1}{q} {\rm dim}_{\mathcal H}(X)$ holds with $q \in (0,1]$ satisfying $(2K)^q = 2$.

\begin{proof}[Proof of Theorem~\ref{T1}]
	Assume that $\rho$ is a metric for which $\mathcal T_\rho = \mathcal T$ and $(X, \rho)$ is geometrically doubling. Then ${\rm dim}_{\mathcal H}(X)$ is finite by Lemma~\ref{L1}. Also, $(X,\rho)$ is separable because the geometrical doubling property forces that for any $M \in \NN$ the whole space $X$ can be covered by countably many balls of radius $2^{-M}$. Thus, Fact~\ref{F5} gives
	\[
	{\rm dim}(X) \leq
	{\rm dim}_{\mathcal H}(X) < \infty = {\rm dim}(X).
	\]
	 This contradicts the existence of $\rho$ with the desired properties.
\end{proof} 

The following remarks highlight why the problem stated in Question~\ref{Q1} was delicate. 
\begin{remark} \label{R1}
	In general, being geometrically doubling is \textbf{not} a topological property. 
\end{remark}

\noindent Indeed, one can change $\rho_\RR$ to make $\RR$ with its natural topology not geometrically doubling. It suffices to take $\rho \coloneqq \log(1 + \rho_\RR)$ and consider the balls $B_\rho(0,n)$ with $n \to \infty$. 

\begin{remark} \label{R2}
	The subspace $\{0, \frac{1}{2} \}^{\omega} \subset \tom$ with the topology inherited from $\tom$ can be made homogeneous in the Coifman--Weiss sense. 
\end{remark}

\noindent Indeed, it suffices to identify $\{0, \frac{1}{2} \}^{\omega}$ with the classical Cantor set $\mathcal C \subset [0,1]$ with the metric $\rho_{\mathcal C}$ obtained by restricting $\rho_\RR$ to $\mathcal C \times \mathcal C$. This can be done through the bijection $\pi \colon \{0, \frac{1}{2} \}^{\omega} \to \mathcal C$ given by $\pi(x) \coloneqq \sum_{n=1}^\infty \frac{4 x_n}{3^n}$. Then the usual Cantor measure $\mu_{\mathcal C}$ on $\mathcal C$ is doubling, since for all $x \in \mathcal C$ and $n \in \NN$ one has $\mu_{\mathcal C}(B_{\rho_{\mathcal C}}(3^{-n})) = 2^{-n}$, see also \cite{WWW}.

\end{document}